\documentclass{amsart}
\usepackage{amsmath}
\usepackage{amssymb}
\usepackage{mathrsfs}
\usepackage{pgf,tikz}
\usepackage{graphicx}
\usepackage{systeme,mathtools}
\usepackage{amsthm,enumitem}
\usepackage{systeme}
\newtheorem{theorem}{Theorem}[section]

\newtheorem{lemma}[theorem]{Lemma}
\newtheorem{corollary}[theorem]{Corollary}
\newtheorem{proposition}[theorem]{Proposition}
\theoremstyle{definition}
\newtheorem{definition}[theorem]{Definition}
\newtheorem{remark}{Remark}
\newtheorem{example}{Example}

\begin{document}
	\title[Lie Algebras of Skew-Symmetric Elements in Simple Leavitt path algebras]{Lie Algebras of Skew-Symmetric Elements in Simple Leavitt path algebras}
	\author[Nguyen Huynh Thao Nhi]{Nguyen Huynh Thao Nhi}
	\author[Huynh Viet Khanh]{Huynh Viet Khanh}
	\address{Department of Mathematics and Informatics, HCMC University of Education, 280 An Duong Vuong Str., Dist. 5, Ho Chi Minh City, Vietnam}
	\email{nhinht.dais034@pg.hcmue.edu.vn; thaonhi01012001@gmail.com}
	\email{khanhhv@hcmue.edu.vn}
	\keywords{leavitt path algebra; Lie algebra.\\ 
	\protect \indent 2020 {\it Mathematics Subject Classification.} 16S88; 17B65; 17B20; 17B30.}
	\maketitle
\begin{abstract} 
        Let $K$ be a field and $E$ be a graph. Let $L_K(E)$ be the Leavitt path algebra of $E$ over $K$ with the standard involution $^\star$. We investigate the set of skew-symmetric elements, $\mathbf{K}_{L_K(E)}=\{x\in L_K(E) : x^{\star}=-x\}$, and show that for any simple $L_K(E)$ containing a cycle, $[\mathbf{K}_{L_K(E)}, \mathbf{K}_{L_K(E)}]\ne\mathbf{K}_{L_K(E)}$. This provides a negative answer to a question posed by Herstein raised in [Bull. Amer. Math. Soc. 67 (1961), 517–531]. 
\end{abstract}
	
\section{Introduction}
Every associative algebra $\mathcal{A}$ over a field $K$ gives rise to a Lie algebra by considering the same vector space structure and defining the Lie bracket as $[a,b] = ab - ba$ for all $a, b \in \mathcal{A}$. The Lie algebra obtained in this way is called the Lie algebra associated with $\mathcal{A}$ and is often denoted by $\mathcal{A}^-$. Suppose further that $\mathcal{A}$ is equipped with an involution $^\star$, that is, a map $^\star: \mathcal{A} \to \mathcal{A}$ satisfying the following conditions:
\begin{enumerate}
	\item $(a+b)^{\star}=a^{\star}+b^{\star}$,
	\item $(ab)^{\star}=b^{\star}a^{\star}$,
	\item $(a^{\star})^{\star}=a$,
\end{enumerate}
for all $a,b\in \mathcal{A}$. Let $Z(\mathcal{A})$ denote the center of $\mathcal{A}$. If $a^\star = a$ for all $a \in Z(\mathcal{A})$, the involution $^\star$ is said to be of the \textit{first kind}. Otherwise, if there exists a nonzero $a \in Z(\mathcal{A})$ such that $a^\star \neq a$, then $^\star$ is said to be of the \textit{second kind}. Define  $\mathbf{K}_{\mathcal{A}}=\{x\in \mathcal{A} : x^{\star}=-x\}$ as the set of \textit{skew-symmetric elements} of $\mathcal{A}$ with respect to $^\star$. It is clear that if $^\star$ is an involution of the first kind, then $\mathbf{K}_{\mathcal{A}}$ is a Lie subalgebra of $\mathcal{A}^-$. For any $K$-subspace $I$ of $\mathcal{A}$, let $[I, I]$ denote ${\rm span}_K\{ab-ba\mid a,b\in I\}$, the $K$-subspace of $\mathcal{A}$ generated by all commutators $ab - ba$. The following question was raised by Herstein (see [7, p. 529] in \cite{Pa_herstein_61}):

\medskip 

{\noindent \textbf{Question 1.}} If $\mathcal{A}$ is a simple ring with an involution of the first kind, is it true that $\mathbf{K}_{\mathcal{A}}=[\mathbf{K}_{\mathcal{A}}, \mathbf{K}_{\mathcal{A}}]$?

\medskip 

Several counterexamples to Question 1 have been identified. For instance, P. H. Lee provided an example of a division ring $D$ with an involution such that $\mathbf{K}_D \neq [\mathbf{K}_D, \mathbf{K}_D]$ in \cite{Pa_Lee82}. Additionally, J. Hennig constructed a simple, locally-finite-dimensional algebra with the same property in \cite{Pa_hennig14}.

Let $E$ be a graph and $K$ a field, and let $L_K(E)$ denote the Leavitt path algebra of $E$ over $K$. The mapping $^\star: L_K(E) \to L_K(E)$, defined by sending $v$ to $v$, $e$ to $e^\star$, and $e^\star$ to $e$ for all $v \in E^0$ and $e \in E^1$, and fixing all elements in $K$, induces an involution of $L_K(E)$. This is called the \textit{standard involution} on $L_K(E)$ (see \cite[page 36]{Bo_abrams_2017}). One can verify that this involution $^\star$ is of the first kind.

In this note, we address Question 1 in the context where $\mathcal{A}$ is the Leavitt path algebra $L_K(E)$. Numerous papers have investigated the structure of $L_K(E)^-$, the Lie algebra associated with $L_K(E)$ (see, e.g., \cite{Pa_khanh24}, \cite{Pa_mesyan-13}, \cite{abrams-mesyan-12}). It is known from \cite[Theorem 2.9.1]{Bo_abrams_2017} that $L_K(E)$ is a simple ring if and only if the following conditions hold:
\begin{enumerate}
	\item The only hereditary and saturated subsets of $E^0$ are $\varnothing$ and $E^0$;
	\item Every cycle in $E$ has an exit.
\end{enumerate}

The primary aim of this note is to study $\mathbf{K}_{L_K(E)}$, the set of skew-symmetric elements of $L_K(E)$ with respect to $^\star$. Additionally, we demonstrate that for any simple $L_K(E)$ containing a cycle, $[\mathbf{K}_{L_K(E)}, \mathbf{K}_{L_K(E)}] \neq \mathbf{K}_{L_K(E)}$, thereby providing negative answers to Question 1.

\medskip 

A \textit{directed graph} is a graph, denoted by $E=(E^0, E^1, r, s)$, consisting of two sets $E^0$ and $E^1$ together with maps $r, s: E^1\to E^0$. The elements of $E^0$ and $E^1$ are called \textit{vertices} and \textit{edges} of $E$ respectively. For $e\in E^1$, we say that $s(e)$ is a \textit{source} and $r(e)$ is a \textit{range} of $e$. In this paper, the word “graph” will always mean “directed graph”. Moreover, concerning a graph as above, we often write $E$ instead of  $E=(E^0, E^1, r, s)$. 

Let $E$ be a graph, $v\in E^0$ and $e\in E^1$. We say that $v$ \textit{emits} $e$ if $s(e)=v$. A vertex $v$ is a \textit{sink} if it emits no edges, while it is an \textit{infinite emitter }if it emits infinitely many edges. A vertex $v$ is said to be \textit{regular} if it is neither a sink nor an infinite emitter. The set of all regular vertices in a graph $E$ is denoted by ${\rm Reg}(E)$. A \textit{finite path} $\mu$ of length $\ell(\mu):=n\ge 1$ is a finite sequence of edges $\mu=e_1e_2\cdots e_n$ with $r(e_i) =s(e_{i+1})$ for all $1\leq i\leq n-1$. We set $s(\mu):=s(e_1)$ and $r(\mu):=r(e_n)$. The set of all finite paths in $E$ is denoted by ${\rm Path}(E)$. Let ${\rm ClPath}(E)$ denote the set of all closed paths in ${\rm Path}(E)$. 

If  $\mu=e_1\cdots e_n\in {\rm Path}(E)$, then we denote the element $e_n^*\cdots e_2^*e_1^*$ by $\mu^*$. If $\ell(\mu)\ge 1$ and $v=s(\mu)=r(\mu)$, then we say that $\mu$ is a \textit{closed path based at} $v$. If, moreover, $s(e_j)\ne v$ for every $j>1$, then we call $\mu$ a \textit{closed simple path based at $v$}. If $\mu=e_1\cdots e_n$ is  a closed path based at $v$ and $s(e_i)\ne s(e_j)$ for every $i\ne j$,
then $\mu$ is called a \textit{cycle based at} $v$. If $\mu=e_1\cdots e_n$ is a cycle based at $v$, then for each $1\le i\le n$, the path $\mu_i=e_ie_{i+1}\cdots e_ne_1\cdots e_{i-1}$ is a cycle based at $s(e_i)$. We call the collection of cycles $\{\mu_i\}$ based at $s(e_i)$ the \textit{cycle of} $\mu$. A \textit{cycle} $c$ is a set of paths consisting of the cycle of $\mu$ for some cycle $\mu$ based at a vertex $v$. An \textit{exit} for a path $\mu=e_1\cdots e_n$ is an edge $e$ such that $s(e)=s(e_i)$ for some $i$ and $e\ne e_i$. 

Let $H$ be a subset of $E^0$. We say that $H$ is \textit{hereditary}  if whenever $u\in H$ and $u\geq v$ for some vertex $v$, then $v\in H$; and $H$ is \textit{saturated} if, for any regular vertex $v$, $r(s^{-1}(v))\subseteq H$ implies $v\in H$. 
\begin{definition}[Leavitt path algebra]\label{definition_1.1}
Let $E$ be an arbitrary graph and $K$ a field. The \textit{Leavitt path algebra of $E$ over $K$}, denoted by $L_K(E)$, is the free associative $K$-algebra generated by $E^0\cup E^1\cup (E^1)^*$, subject to the following relations:
	\begin{enumerate}[]
		\item[(V)] $vv'=\delta_{v,v'}v$ for all $v,v'\in E^0$,
		\item[(E1)] $s(e)e=er(e)=e$ for all $e\in E^1$,
		\item[(E2)] $r(e)e^*=e^*s(e)=e^*$ for all $e\in E^1$,
		\item[(CK1)] $e^*f=\delta_{e,f}r(e)$ for all $e,f\in E^1$, and
		\item[(CK2)] $v=\displaystyle\sum_{\{e\in E^1|s(e)=v\}}ee^*$ for every $v\in{\rm Reg}(E)$.
	\end{enumerate}
\end{definition}

\begin{definition}[Cohn path algebra]
		Let $E$ be an arbitrary graph and $K$ any field. The \textit{Cohn path algebra of $E$ over $K$}, denoted by $C_K(E)$, is the free associative $K$-algebra generated by the set $E^0\cup E^1\cup (E^1)^*$, subject to the relations given in (V), (E1), (E2), and (CK1) in Definition \ref{definition_1.1}. 
\end{definition}

\begin{remark}\label{Remark_1}
		Let $E$ be a graph and $K$ a field. Let $I$ be the ideal of the $C_K(E)$ generated by the set 
		$$
		\Big\lbrace v-\sum_{e\in s^{-1}(v)}ee^*\;|\;  v\in {\rm Reg(E)}\Big\rbrace .
		$$
		Then 
		$$
		L_K(E)\cong C_K(E)/I
		$$ 
		as $K$-algebras.
\end{remark} 	
		Let $E$ be a graph and $K$ a field. For simplicity reason, throughout this note, we always let $\mathcal{L}$ and $\mathcal{C}$ denote $L_K(E)$ and $C_K(E)$ respectively.  The set of all integers and positive integers will be denoted by $\mathbb{Z}$ and $\mathbb{Z}_+$, respectively.
\section{A discription of $\mathbf{K}_{\mathcal{L}}$} 
		Although various bases for $\mathcal{L}$ can be identified (see e.g., \cite{Bo_abrams_2017}), consider the set
		$$
     	\mathscr{B}=\{\alpha\beta^*\mid \alpha,\beta\in{\rm Path}(E), r(\alpha)=r(\beta)\}.
     	$$
     	While $\mathcal{L}$ is generated as a $K$-vector space by $\mathscr{B}$, this set is not necessarily a basis for $\mathcal{L}$. However, according to \cite[Proposition 1.5.6]{Bo_abrams_2017}, $\mathscr{B}$ is indeed a $K$-basis for $\mathcal{C}$. Consequently, computations involving the elements of $\mathcal{C}$ are simplified when using this basis. Every element $x \in \mathcal{C}$ can be uniquely expressed as $x = \sum_{i=1}^n k_i \alpha_i \beta_i^*$, where $k_i \in K\backslash\{0\}$ (the nonzero elements of $K$), and $\alpha_i, \beta_i \in \text{Path}(E)$ with $r(\alpha_i) = r(\beta_i)$ for each $i \in\{1,\dots,n\}$. Upon rearrangement, $x$ can be written in the form
\begin{equation}\label{uniqe_representation}
     x=\sum_{i=1}^{n} \Big(a_i\gamma_i\lambda_i^*+b_i\lambda_i\gamma_i^*\Big)+\sum_{j=1}^kc_jp_jp_j^*,
\end{equation}
		where $a_i, b_i, c_j \in K$, $\gamma_i, \lambda_i, p_j \in \text{Path}(E)$, with $\gamma_i \neq \lambda_i$ and $\gamma_i \lambda_i^* \neq \lambda_{i'} \gamma_{i'}^*$ for all $1\leq i\ne i' \leq n$. (Note that some of the $a_i$, $b_i$, and $c_j$ may be zero.)

\begin{proposition}\label{proposition_2.1}
	Let $E$ be a graph and $K$ a field. The following assertions hold:
\begin{enumerate}[font=\normalfont]
		\item[(i)] If ${\rm char}(K)\ne2$, then $\mathbf{K}_{\mathcal{C}}$ is equal to 
     	$$
     	{\rm span}_K\{\gamma\lambda^*-\lambda\gamma^*\mid\gamma,\lambda\in {\rm Path}(E),\gamma\ne\lambda, r(\gamma)=r(\lambda)\}.
     	$$	
  		\item[(ii)] If ${\rm char}(K)=2$, then $\mathbf{K}_{\mathcal{C}}$ is equal to 
    	$$
     	{\rm span}_K\{pp^*, \gamma\lambda^*+\lambda\gamma^*\mid p,\gamma,\lambda\in {\rm Path}(E),\gamma\ne\lambda, r(\gamma)=r(\lambda)\}.
     	$$
\end{enumerate}
\end{proposition}
\begin{proof}
		Let $x\in\mathbf{K}_{\mathcal{C}}$. Write
		$$
		x=\sum_{i=1}^{n} (a_i\gamma_i\lambda_i^*+b_i\lambda_i\gamma_i^*)+\sum_{j=1}^kc_jp_jp_j^*,
		$$
		where $a_i, b_i, c_j\in K$, $\gamma_i\ne\lambda_i$, and $\gamma_i\lambda_i^*\ne \lambda_{i'}\gamma_{i'}^*$ for all $1\leq i\ne i'\leq n$. It follows that 
		$$
		x^{\star}=\displaystyle \sum_{i=1}^{n}\Big(a_i\lambda_i \gamma_i^*+b_i\gamma_i \lambda_i^*\Big)+\sum_{j=1}^kc_jp_jp_j^*.
		$$
		As $x^{\star}+x=0$, we get
		\begin{equation}\label{equation a_i b_i}
		\sum_{i=1}^{n}\Big(a_i+b_i\Big)\gamma_i\lambda_i^*+\sum_{i=1}^{n}\Big(a_i+b_i\Big)\lambda_i\gamma_i^*+\sum_{i=1}^k2c_jp_jp_j^*=0.
		\end{equation}
	Because the elements  $\gamma_1\lambda_1^*,\dots,\gamma_n\lambda_n^*, \lambda_1\gamma_1^*,\dots,\lambda_n\gamma_n^*, p_1p_1^*,\dots,p_kp_k^*$ are pairwise distinct, we obtain that $a_i+b_i=0$ and $2c_j=0$ for all $i,j$. We divide our situation into two possible cases:
            
		\bigskip 
            
		\textit{Case 1. ${\rm char}(K)\ne2$}. In this case, we have $a_i=-b_i$ and $c_j=0$ for all $i,j$. This means that
		$$
		x=\sum_{i=1}^{n} a_i\Big(\gamma_i\lambda_i^*-\lambda_i\gamma_i^*\Big).
		$$ 
            
		\bigskip 
            
		\textit{Case 2. ${\rm char}(K)=2$}.  It follows that $a_i=b_i$ for all $i$. Consequently, we get
		$$
		x=\sum_{i=1}^{n} a_i\Big(\gamma_i\lambda_i^*+\lambda_i\gamma_i^*\Big)+\sum_{j=1}^kc_jp_jp_j^*.
		$$
         Conversely, such elements $x$ with presentations in \textit{Case 1} and \textit{Case 2} clearly belong to $\mathbf{K}_{\mathcal{C}}$. Thus, the proposition is now proved.
\end{proof}
\begin{lemma}[{\cite[Lemma 9]{abrams-mesyan-12}}]\label{lemma_2.2}
Let $K$ be a field and $E$ a graph. Let $v \in E^0$ be a regular vertex, and let $y$ denote the element $v-\sum_{\left\{e \in E^1 \mid s(e)=v\right\}} e e^*$ of the ideal $I$ of $\mathcal{C}$ described in Remark \ref{Remark_1}.
		\begin{enumerate}[font=\normalfont]
			\item[(i)] If $p \in \operatorname{Path}(E) \backslash E^0$, then $y p=0$.
			\item[(ii)]  If $q \in \operatorname{Path}(E) \backslash E^0$, then $q^* y=0$.
		\end{enumerate}
\end{lemma}
         
\begin{lemma}\label{lemma_2.3}
Let $E$ be a graph and $K$ a field. Let $I$ be the ideal of $\mathcal{C}$ as defined in Remark \ref{Remark_1}. The following assertions hold:
		\begin{enumerate}[font=\normalfont]
                \item[(i)] $I \cap \mathbf{S}_{\mathcal{C}}$ is equal to
                 $$
                 {\rm span}_K\Big\lbrace\gamma\gamma^*-\sum_{e\in s^{-1}(v)}\gamma ee^*\gamma^*\;|\; \gamma\in {\rm Path}(E), r(\gamma)=v\in {\rm Reg}(E)\Big\rbrace  .
                 $$
                 \item[(ii)] If ${\rm char} (K)=2$, then $I \cap \mathbf{K}_{\mathcal{C}}=I \cap \mathbf{S}_{\mathcal{C}}$. If ${\rm char} (K)\ne2$, then $I \cap \mathbf{K}_{\mathcal{C}}=0$.
		\end{enumerate}
\end{lemma}
\begin{proof}
            Let $0\ne x\in I$. In view of Lemma \ref{lemma_2.2}, we get
            $$
                x=\sum_{i=1}^nk_i\gamma_i  \Big(v_i-\sum_{e\in s^{-1}(v_i)}ee^*\Big)\rho_i^*,
            $$
            where $k_i\in K\backslash\{0\}$, $\gamma_i,\rho_i\in {\rm Path}(E)$ with $r(\gamma_i)=r(\rho_i)=v_i\in{\rm Reg}(E)$ for $1\leq i\leq n$. Or, equivalently, 
            $$
                x=\sum_{i=1}^n k_i\Big(\gamma_i\rho_i^*-\sum_{e\in s^{-1}(v_i)}\gamma_iee^*\rho_i^*\Big).
            $$
            (i) Assume that $x\in \mathbf{S}_{\mathcal{C}}$. Then, we have $x-x^\star=0$; that is,
            $$
                \sum_{i=1}^n k_i\Big(\gamma_i\rho_i^*-\rho_i\gamma_i^*-\sum_{e\in s^{-1}(v_i)}(\gamma_iee^*\rho_i^*-\rho_iee^*\gamma_i^*)\Big)=0.
            $$
            As $v_i\ne v_j$ for all $1\leq i\ne j\leq n$ and all the $k_i$'s are non-zero, the above equation implies that $\rho_i=\gamma_i$ for all $i$. In other words, we have 
            $$
            x=\displaystyle \sum_{i=1}^n k_i\Big(\gamma_i\gamma_i^*-\sum_{e\in s^{-1}(v_i)}\gamma_iee^*\gamma_i^*\Big).
            $$
            Conversely, it is clear that such an element belongs to $I\cap \mathbf{S}_{\mathcal{C}}$.

            \medskip

            (ii) If ${\rm char}(K)=2$, then $\mathbf{K}_{\mathcal{C}}=\mathbf{S}_{\mathcal{C}}$, from which it follows that $I\cap \mathbf{S}_{\mathcal{C}}=I \cap \mathbf{K}_{\mathcal{C}}.$ Now we only need to consider the case ${\rm char}(K)\ne 2$. Deny the conclusion. Assume that $0\ne x\in \mathbf{K}_{\mathcal{C}}$. It follows that $x+x^\star=0$, which implies that
            \begin{equation}\label{equation_diff}
                \sum_{i=1}^n k_i\Big(\gamma_i\rho_i^*+\rho_i\gamma_i^*-\sum_{e\in s^{-1}(v_i)}(\gamma_iee^*\rho_i^*+\rho_iee^*\gamma_i^*)\Big)=0.
            \end{equation}
            As $v_i\ne v_j$ for $i\ne j$, if there exists $i_0\in\{1\leq i\leq n\}$ such that $\gamma_{i_0}\ne \rho_{i_0}$, then the summand $k_{i_0}\gamma_{i_0}\rho_{i_0}^*$ in the equation (\ref{equation_diff}) is not eliminated, which implies that $k_{i_0}=0$, a contradiction. Therefore, we get that $\rho_i=\gamma_i$ for all $i$. Then, the equation (\ref{equation_diff}) becomes:
            \begin{equation}\label{equation_same}
                2\displaystyle \sum_{i=1}^n k_i\Big(\gamma_i\gamma_i^*-\sum_{e\in s^{-1}(v_i)}\gamma_iee^*\gamma_i^*\Big)=0.
            \end{equation}
            As $v_i\ne v_j$ for $i\ne j$, the equation (\ref{equation_same}) implies that $2k_i=0$, and hence $k_i=0$ for all $i\in \{1\leq i\leq n\}$. It follows that $x=0$, which is a contradiction. Therefore, we get that $I\cap\mathbf{K}_{\mathcal{C}}=0$.             
\end{proof}

\begin{proposition}\label{propostion_2.4}
            Let $E$ be a graph and $K$ a field. 
            \begin{enumerate}[font=\normalfont]
                \item[(i)] If ${\rm char}(K)\ne2$, then $\mathbf{K}_{\mathcal{L}}$ is equal to 
                $$
                {\rm span}_K\{\gamma\lambda^*-\lambda\gamma^*\;|\;\gamma,\lambda\in {\rm Path}(E), r(\gamma)=r(\lambda), \gamma\ne\lambda\}.
                $$
                \item[(ii)] If ${\rm char}(K)=2$, then $\mathbf{K}_{\mathcal{L}}$ is equal to  
                $$
                {\rm span}_K\{pp^*,\gamma\lambda^*+\lambda\gamma^*\;|\;p,\gamma,\lambda\in {\rm Path}(E), r(\gamma)=r(\lambda), \gamma\ne\lambda\}.
                $$
			\end{enumerate}
\end{proposition}
\begin{proof}
            Viewing $\mathcal{L}$ as $\mathcal{C}/I$, a non-zero element $\bar{x}\in \mathcal{L}$ can be written in the form
            $$
            \bar{x}=\sum_{i=1}^n \Big(a_i\gamma_i\lambda_i^*+b_i\lambda_i\gamma_i^*\Big)+\sum_{j=1}^m c_j p_j p_j^*+I,
            $$
            where $a_i, b_i, c_j\in K$, $\gamma_i\ne\lambda_i$, and $\gamma_i\lambda_i^*\ne \lambda_{i'}\gamma_{i'}^*$ for all $1\leq i\ne i'\leq n$. Assume that $\bar{x}\in \mathbf{K}_{\mathcal{L}}$. It follows that $\bar{x}+\bar{x}^\star=\bar{0}$, which implies that
            $$
            \displaystyle\sum_{i=1}^n(a_i+b_i)\Big(\gamma_i\lambda_i^*+\lambda_i\gamma_i^*\Big)+2\sum_{j=1}^m c_j p_j p_j^*+I=\bar{0}.
            $$
            Let $y$ denote the element $\displaystyle\sum_{i=1}^n(a_i+b_i)\Big(\gamma_i\lambda_i^*+\lambda_i\gamma_i^*\Big)+2\sum_{j=1}^m c_j p_j p_j^*\in \mathcal{C}$. Then, we have $y\in I$ in $\mathcal{C}$. Moreover, it is easy to check that  $y\in \mathbf{S}_{\mathcal{C}}$. In other words, we have $y\in \mathbf{S}_{\mathcal{C}}\cap I$. By Lemma \ref{lemma_2.3}, we obtain that 
            $$
            \sum_{i=1}^n(a_i+b_i)\Big(\gamma_i\lambda_i^*+\lambda_i\gamma_i^*\Big)+2\sum_{j=1}^m c_j p_j p_j^*=\sum_{k=1}^s d_k \Big(\rho_k\rho_k^*-\sum_{e\in s^{-1}(v_k)}\rho_k ee^*\rho_k^* \Big),
            $$
            where $d_1,d_2,\dots,d_s\in K$ and $v_1,v_2,\dots,v_s\in {\rm Reg}(E)$. Since $\gamma_i\ne\lambda_i$, and $\gamma_i\lambda_i^*\ne \lambda_{i'}\gamma_{i'}^*$ for all $1\leq i\ne i'\leq n$, the previous equation implies that 
            $a_i=-b_i$ for all $i\in\{1,\dots,n\}$, and so 
            \begin{equation}\label{equation_reduced}
                2\sum_{j=1}^m c_j p_j p_j^*=\sum_{k=1}^s d_k\Big(\rho_k\rho_k^*-\sum_{e\in s^{-1}(v_k)}\rho_k ee^*\rho_k^* \Big)\in I,
            \end{equation}
            and
            $$
            \bar{x}=\displaystyle \sum_{i=1}^n a_i\Big(\gamma_i \lambda_i^*-\lambda_i\gamma_i^*\Big)+\displaystyle \sum_{j=1}^m c_j p_j p_j^*+I.
            $$
            If ${\rm char}(K)=2$, then such an element $\bar{x}$ with the form as above clearly belongs to $\mathbf{K}_{\mathcal{L}}$. This proves (ii). If ${\rm char}(K)\ne2$, then the equation (\ref{equation_reduced}) implies that $\sum_{j=1}^m c_j p_j p_j^*\in I$. It follows that $\bar{x}=\displaystyle \sum_{i=1}^n a_i\Big(\gamma_i \lambda_i^*-\lambda_i\gamma_i^*\Big)+I$. Again, such an element also belongs to $\mathbf{K}_{\mathcal{L}}$, proving (i).
\end{proof}
\section{Negative answers to Herstein's Question}
        A comprehensive description of the elements of $[\mathcal{L}, \mathcal{L}]$ was provided by Mesyan in \cite{Pa_mesyan-13}. Using this description, along with our results from the previous section, we present in this section a characterization of $\mathbf{K}_{\mathcal{L}}\cap [\mathcal{L}, \mathcal{L}]$. This characterization enables us to provide negative answers to Question 1. First, we begin by recalling the following definitions:
        
\begin{definition}[{\cite[Definition 1]{Pa_mesyan-13}}]\label{definition_3.1}
	Let $E$ be a graph, and for each $v\in E^0$ let $\epsilon_v\in\mathbb{Z}^{(E^0)}$ denote the element with $1$ in the $v$-th coordinate and zeros elsewhere. If $v\in E^0$ is a regular vertex, for all $a_{vu}$ denote the number of edges $e\in E^1$ such that $s(e)=v$ and $r(e)=u$. Define
	$$
	B_v=(a_{vu})_{u\in E^0}-\epsilon_v\in \mathbb{Z}^{(E^0)}.
	$$
	On the other hand, let 
	$$
	B_v=(0)_{u\in E^0}\in \mathbb{Z}^{(E^0)},
	$$
	if $v$ is not a regular vertex.
\end{definition}
        
\begin{definition}[{\cite[Definition 11]{Pa_mesyan-13}}]
		Let $E$ be a graph and $p,q \in {\rm{ClPath}}(E)$. We write $p \sim q$ if there exist paths $x,y \in {\rm{Path}}(E)$ such that $p = xy$ and $q= yx$.
\end{definition}
        Also, it was proved in \cite[Lemma 12]{Pa_mesyan-13} that $\sim$ is an equivalence relation on the elements of ${\rm ClPath}(E)$. We also recall here the description of $[\mathcal{L}, \mathcal{L}]$ given by Mesyan:
        \begin{theorem}[{\cite[Theorem 15]{Pa_mesyan-13}}]\label{theorem_3.3}
        	Let $E$ be a graph and $K$ a field. Let $x\in\mathcal{L}$ be an arbitrary element. Write
        	$$
        	x=\displaystyle \sum_{i=1}^{k} a_{i} p_{i} p_{i}^{*}+\displaystyle \sum_{i=1}^{l} b_{i} q_{i} t_{i}^{*}+\displaystyle \sum_{i=1}^m \sum_{j=1}^{m_i} c_{i j} x_{i j} y_{i j} x_{i j}^*+\displaystyle \sum_{i=1}^n \sum_{j=1}^{n_i} d_{i j} z_{i j} w_{i j}^* z_{i j}^*,
        	$$
        	for some $k,l,m,m_i,n,n_i\in\mathbb{Z}_+$,  $a_i,b_i, c_{ij}, d_{ij}\in K$ and $p_i,q_i,t_i, x_{ij},z_{ij}\in{\rm Path}(E)$, and $y_{ij},w_{ij}\in {\rm ClPath}(E)\backslash E^0$, where for each $i$ and $u\in{\rm ClPath}(E)$, $q_i\ne t_iu$ and $t_i\ne q_iu$, and where $y_{ij}\sim y_{i'j'}$ if and only if $i=i'$ if and only if $w_{ij}\sim w_{i'j'}$. Also, for each $v\in E^0$, let $B_v$ and $\epsilon_v$ be as in Definition \ref{definition_3.1}. \\
        	Then $x\in [\mathcal{L}, \mathcal{L}]$ if and only if the following conditions are satisfied:
        	\begin{enumerate}[font=\normalfont]
        		\item $\sum_{i=1}^ka_i\epsilon_{r(p_i)}\in{\rm span}_K\{B_v|v\in E^0\}$,
        		\item $\sum_{j=1}^{m_i}c_{ij}=0$ for each $i\in\{1,\dots,m\}$,
        		\item $\sum_{j=1}^{n_i}d_{ij}=0$ for each $i\in\{1,\dots,n\}$.
        	\end{enumerate}
        \end{theorem}
\begin{lemma}\label{lemma_3.4}
		Let $E$ be a graph and $K$ a field.  Let $x$ be an element of $\mathcal{C}$ which is uniquely written as $x=\sum_{i=1}^lb_it_iq_i^*$, where $l\in\mathbb{Z}_+$, $b_i\in K\backslash\{0\}$ and $q_i,t_i\in{\rm Path}(E)$ with $r(q_i)=r(t_i)$ for all $1\leq i\leq n$. If $x\in \mathbf{K}_{\mathcal{C}}$, then $l=2l'$ for some $l'\in\mathbb{Z}_+$ and  there exist $b_{i_1}, \dots, b_{i_{l'}}\in \{b_1,\dots b_l\}$ and $t_{i_1}q_{i_1}^*, \dots, t_{i_{l'}}q_{i_{l'}}^*\in\{t_1q_1^*,\dots, t_lq_l^*\}$ such that  $x=\sum_{j=1}^{l'}b_{i_j}(t_{i_j}q_{i_j}^*-q_{i_j}t_{i_j}^*)$.
\end{lemma}
\begin{proof}
		As $x+x^\star=0$, we get
		\begin{align}\label{equation_2}
			\displaystyle \sum_{i=1}^{l} b_{i} t_{i} q_i^*+ \displaystyle \sum_{i=1}^{l}b_iq_i t_i^*=0.
		\end{align}
		As the  the element $t_iq_i^*$'s are pairwise distinct and all the $b_i$'s are non-zero, the equation (\ref{equation_2}) implies that for each $i\in\{1,\dots,l\}$, there exists $i\ne j\in\{1,\dots,l\}$ such that $b_it_iq_i^*+b_jq_jt_j^*=0$. It follows $l=2l'$ for some $l'\in \mathbb{Z}_+$, and that $b_i=-b_j$, $t_i=q_j$ and $q_i=t_j$. Thus, the conclusion of the lemma follows immediately. 
\end{proof}

\begin{lemma}\label{lemma_3.5}
		Let $E$ be a graph and $K$ a field.  Let $x$ be an element of $\mathcal{C}$ which is uniquely written as
		$$
		x=\displaystyle \sum_{i=1}^m \sum_{j=1}^{m_i} c_{i j} x_{i j} y_{i j} x_{i j}^*+\displaystyle \sum_{i=1}^n \sum_{j=1}^{n_i} d_{i j} z_{i j} w_{i j}^* z_{i j}^*,
		$$
	 	where $m, n, m_i, n_i\in\mathbb{Z}_+$, $c_{ij},d_{ij}\in K\backslash\{0\}$ and $x_{ij},z_{ij}\in{\rm Path}(E)$, $y_{ij},w_{ij}\in {\rm ClPath}(E)\backslash E^0$, and where $y_{ij}\sim y_{i'j'}$ if and only if $i=i'$ if and only if $w_{ij}\sim w_{i'j'}$. If $x\in \mathbf{K}_{\mathcal{C}}$, then $m=n$ and after a suitable re-indexing, we have that for each $i\in\{1,\dots,m\}$, $m_i=n_i$ and $d_{ij}=-c_{i j},\, z_{ij}=x_{i j}$ and $w_{ij}=y_{i j}$ for all $1\leq j\leq m_i$. In other words, we have 
	 	$$
	 	x= \displaystyle \sum_{i=1}^m \displaystyle\sum_{j=1}^{m_i} c_{ij} \Big( x_{ij} y_{ij} x_{ij}^*- x_{ij} y_{ij}^* x_{ij}^*\Big).
	 	$$ 
\end{lemma}
\begin{proof}
		As $x+x^\star=0$, we get
		\begin{align}\label{equation_3}
			\displaystyle \sum_{i=1}^m \sum_{j=1}^{m_i} c_{i j} \Big( x_{i j} y_{i j} x_{i j}^*+x_{ij}y_{ij}^* x_{ij}^*\Big)+\displaystyle \sum_{i=1}^n \sum_{j=1}^{n_i} d_{i j} \Big(z_{i j} w_{i j}^* z_{i j}^*+z_{ij}w_{ij}z_{ij}^*\Big)=0.
		\end{align}
		For each $i\in\{1,\dots,m\}$ and $j\in\{1,\dots,m_i\}$, and for each $k\in\{1,\dots,n\}$ and $l\in\{1,\dots,n_k\}$ we have 
		$$
		c_{ij}x_{ij}y_{ij}x_{ij}^*+d_{kl}z_{kl}w_{kl}z_{kl}^* =0 \Longleftrightarrow \begin{cases} c_{ij}=-d_{kl} \\
			x_{ij}=z_{kl}\\
			y_{ij}=w_{kl}
		\end{cases} \Longleftrightarrow c_{ij}x_{ij}y_{ij}^*x_{ij}^*+d_{kl}z_{kl}w_{kl}^*z_{kl}^* =0.
		$$
		Therefore, the conclusion follows immediately. 
\end{proof}

\begin{theorem}\label{theorem_3.6}
Let $E$ be a graph and $K$ a field. Let $x\in\mathcal{L}$.
		\begin{enumerate}[font=\normalfont]
                \item[(i)] If ${\rm char}(K)\ne2$, then $x\in \mathbf{K}_{\mathcal{L}}$ if and only if there exist $l, m, m_i\in\mathbb{Z}_+$,  $b_i, c_{ij}\in K$, $q_i,t_i, x_{ij}\in{\rm Path}(E)$, $y_{ij}\in {\rm ClPath}(E)\backslash E^0$ such that
                $$
                x=\sum_{i=1}^{l}b_i(t_iq_i^*-q_it_i^*)+\displaystyle \sum_{i=1}^m \displaystyle\sum_{j=1}^{m_i} c_{ij} \Big( x_{ij} y_{ij}^* x_{ij}^*- x_{ij} y_{ij} x_{ij}^*\Big),$$
             	where for each $i$ and $u\in{\rm ClPath}(E)$, $q_i\ne t_iu$ and $t_i\ne q_iu$, and where $y_{ij}\sim y_{i'j'}$ if and only if $i=i'$.
                \item[(ii)] If ${\rm char}(K)=2$, then $x\in \mathbf{K}_{\mathcal{L}}$ if and only if  there exist $k, l, m, m_i\in\mathbb{Z}_+$ and  $p_i, b_i, c_{ij}\in K$, $q_i,t_i, x_{ij}\in{\rm Path}(E)$, and $y_{ij}\in {\rm ClPath}(E)\backslash E^0$ such that
                $$
                x=\displaystyle \sum_{i=1}^k a_i p_i p_i^* +\sum_{i=1}^{l}b_i(t_iq_i^*-q_it_i^*)+ \displaystyle \sum_{i=1}^m \displaystyle\sum_{j=1}^{m_i} c_{ij} \Big( x_{ij} y_{ij}^* x_{ij}^*- x_{ij} y_{ij} x_{ij}^*\Big),
                $$
                where $\sum_{i=1}^ka_i\epsilon_{r(p_i)}\in{\rm span}_K\{B_v|v\in E^0\}$, and for each $i$ and $u\in{\rm ClPath}(E)$, $q_i\ne t_iu$ and $t_i\ne q_iu$, and where $y_{ij}\sim y_{i'j'}$ if and only if $i=i'$.
		\end{enumerate}
		
\end{theorem}
\begin{proof}
           Let $x\in\mathcal{L}$. Viewing $x$ as an element of $\mathcal{C}/I$, we can write 
           $$
           x=\displaystyle \sum_{i=1}^{k} a_{i} p_{i} p_{i}^{*}+\displaystyle \sum_{i=1}^{l} b_{i} q_{i} t_{i}^{*}+\displaystyle \sum_{i=1}^m \sum_{j=1}^{m_i} c_{i j} x_{i j} y_{i j} x_{i j}^*+\displaystyle \sum_{i=1}^n \sum_{j=1}^{n_i} d_{i j} z_{i j} w_{i j}^* z_{i j}^*+I,
           $$
           which satisfies condition (1), (2) and (3) of Theorem \ref{theorem_3.3}. As $x \in \mathbf{K}_{\mathcal{L}} $, we get $x^\star +x + I=I$, which means the element
           \begin{align*}
            y:=\displaystyle \sum_{i=1}^k 2a_i p_i p_i^*+ \displaystyle \sum_{i=1}^{l} b_{i} \Big(t_{i} q_i^*+ q_i t_i^*\Big) +\displaystyle \sum_{i=1}^m  \sum_{j=1}^{m_i} c_{i j} \Big( x_{i j} y_{i j} x_{i j}^*+x_{ij} y_{ij}^* x_{ij}^*\Big)\\+\displaystyle \sum_{i=1}^n \sum_{j=1}^{n_j} d_{i j} \Big(z_{i j} w_{i j}^* z_{i j}^*+z_{ij}w_{ij}z_{ij}^*\Big)\in I\text{ in } \mathcal{C}.
           \end{align*}
           Then, it is clear that $y\in \mathbf{S}_{\mathcal{C}}\cap I$. We have two possible cases:

            \bigskip 
            
            \textit{Case 1. $y=0$.} In other words, we have
            \begin{align*}
                \displaystyle \sum_{i=1}^k 2a_i p_i p_i^*+ \displaystyle \sum_{i=1}^{l} b_{i} \Big(t_{i} q_i^*+ q_i t_i^*\Big) +\displaystyle \sum_{i=1}^m  \sum_{j=1}^{m_i} c_{i j} \Big( x_{i j} y_{i j} x_{i j}^*+x_{ij} y_{ij}^* x_{ij}^*\Big)\\+\displaystyle \sum_{i=1}^n \sum_{j=1}^{n_i} d_{i j} \Big(z_{i j} w_{i j}^* z_{i j}^*+z_{ij}w_{ij}z_{ij}^*\Big)=0.
            \end{align*}
            Put 
            \begin{align*}
            	&{\mathbf A}=\{p_ip_i^*\mid 1\leq i\leq n\}\\
            	&{\mathbf B}=\{t_iq_i^*, q_it_i^*\mid 1\leq i\leq l\}\\
            	&{\mathbf C}=\cup_{i=1}^m(\cup_{j=1}^{m_i}\{ x_{ij}y_{ij}x_{ij}^*\})\cup\cup_{i=1}^n(\cup_{j=1}^{n_i}\{ z_{ij}w_{ij}z_{ij}^*\}).
            \end{align*}
            Because these subsets are pairwise disjoint and contained in the basic $\mathscr{B}$ of $\mathcal{C}$, the previous equation implies that 
            \begin{align}\label{equation_first}
            	\displaystyle \sum_{i=1}^k 2a_i p_i p_i^*=0, 
            \end{align}
             \begin{align}\label{equation_second}
            	\displaystyle \sum_{i=1}^{l} b_{i} \Big(t_{i} q_i^*+ q_i t_i^*\Big)=0,
            \end{align}
            \begin{align}\label{equation_third}
            	\displaystyle \sum_{i=1}^m \sum_{j=1}^{m_i} c_{i j} \Big( x_{i j} y_{i j} x_{i j}^*+x_{ij}y_{ij}^* x_{ij}^*\Big)+\displaystyle \sum_{i=1}^n \sum_{j=1}^{n_i} d_{i j} \Big(z_{i j} w_{i j}^* z_{i j}^*+z_{ij}w_{ij}z_{ij}^*\Big)=0.
            \end{align}
            Because the elements of ${\mathbf A}$ are pairwise distinct, the equation \ref{equation_first} implies that $2 a_i=0$ for  $1\leq i\leq k$. Consider the equation (\ref{equation_second}).  In view of Lemma \ref{lemma_3.4}, $t=2t'$ for some $t'\in\mathbb{N}$ and  there exist $b_{i_1}, \dots, b_{i_{l'}}\in \{b_1,\dots b_l\}$ and $t_{i_1}q_{i_1}^*, \dots, t_{i_{l'}}q_{i_{l'}}^*\in\{t_1q_1^*,\dots, t_lq_l^*\}$ such that  $\displaystyle \sum_{i=1}^{l} b_{i} q_{i} t_{i}^{*}=\sum_{j=1}^{l'}b_i'(t_{i_j}q_{i_j}^*-q_{i_j}t_{i_j}^*)$.
			
			Now, consider the equation (\ref{equation_third}). By Lemma \ref{lemma_3.5}, we get that 
			$$
			\displaystyle \sum_{i=1}^m \sum_{j=1}^{m_i} c_{i j} x_{i j} y_{i j} x_{i j}^*+\displaystyle \sum_{i=1}^n \sum_{j=1}^{n_i} d_{i j} z_{i j} w_{i j}^* z_{i j}^*=\displaystyle \sum_{i=1}^m \displaystyle\sum_{j=1}^{m_i} c_{ij} \Big( x_{ij} y_{ij} x_{ij}^*- x_{ij} y_{ij}^* x_{ij}^*\Big).
			$$
			Therefore, 
             $$
             x= \sum_{i=1}^{k} a_{i} p_{i} p_{i}^{*}+\sum_{j=1}^{l'}b_i'(t_{i_j}q_{i_j}^*-q_{i_j}t_{i_j}^*)+\displaystyle \sum_{i=1}^m \displaystyle\sum_{j=1}^{m_i} c_{ij} \Big( x_{ij} y_{ij} x_{ij}^*- x_{ij} y_{ij}^* x_{ij}^*\Big)+I.
             $$ 
            
            \medskip
            
            If ${\rm char}(K)=2$ then we have $(p_i p_i^*)^\star=p_i p_i^*=-p_i p_i^*$, which means $p_i p_i^*$ is in  $\mathbf{K}_{\mathcal{L}}$. If ${\rm char}(K)\ne2$, then $a_i=0$ for all $i$, and again we have $x\in \mathbf{K}_{\mathcal{L}}\cap [\mathcal{L}, \mathcal{L}]$.

            \medskip
            
            \textit{Case 2. $y\ne0$.} By Lemma \ref{lemma_2.3}, we obtain that
            
            \begin{align*}
                 \displaystyle \sum_{i=1}^k 2a_i p_i p_i^*+ \displaystyle \sum_{i=1}^{l} b_{i} \Big(t_{i} q_i^*+ q_i t_i^*\Big) +\displaystyle \sum_{i=1}^m  \sum_{j=1}^{m_i} c_{i j} \Big( x_{i j} y_{i j} x_{i j}^*+x_{ij} y_{ij}^* x_{ij}^*\Big)\\+\displaystyle \sum_{i=1}^n \sum_{j=1}^{n_i} d_{i j} \Big(z_{i j} w_{i j}^* z_{i j}^*+z_{ij}w_{ij}z_{ij}^*\Big)=\sum_{i=1}^t k_i(\rho_i\rho_i^*-\sum_{e\in s^{-1}(v_i)}\rho_i ee^*\rho_i^* ),
            \end{align*}
            for some $k_1,k_2,\dots,k_m\in K\backslash\{0\}$ and $v_1,v_2,\dots,v_m\in {\rm Reg}(E)$. 
            By the same arguments as in Case 1, we get 
            \begin{align}\label{equation_first_1}
            	\displaystyle \sum_{i=1}^k 2a_i p_i p_i^*=\sum_{i=1}^t k_i(\rho_i\rho_i^*-\sum_{e\in s^{-1}(v_i)}\rho_i ee^*\rho_i^* ).
            \end{align}
            \begin{align}\label{equation_second_1}
            	\displaystyle \sum_{i=1}^{l} b_{i} \Big(t_{i} q_i^*+ q_i t_i^*\Big)=0,
            \end{align}
            \begin{align}\label{equation_third_1}
            	\displaystyle \sum_{i=1}^m \sum_{j=1}^{m_i} c_{i j} \Big( x_{i j} y_{i j} x_{i j}^*+x_{ij}y_{ij}^* x_{ij}^*\Big)+\displaystyle \sum_{i=1}^n \sum_{j=1}^{n_i} d_{i j} \Big(z_{i j} w_{i j}^* z_{i j}^*+z_{ij}w_{ij}z_{ij}^*\Big)=0.
            \end{align}
           Again, repeat arguments used in Case 1, we obtain 
          	$$
          	x= \sum_{i=1}^{k} a_{i} p_{i} p_{i}^{*}+\sum_{j=1}^{l'}b_i'(t_{i_j}q_{i_j}^*-q_{i_j}t_{i_j}^*)+\displaystyle \sum_{i=1}^m \displaystyle\sum_{j=1}^{m_i} c_{ij} \Big( x_{ij} y_{ij} x_{ij}^*- x_{ij} y_{ij}^* x_{ij}^*\Big)+I.
          	$$        
            If ${\rm char}(K)=2$ then we have $p_i p_i^*$ is in  $\mathbf{K}_{\mathcal{L}}$. It follows that  $x\in \mathbf{K}_{\mathcal{L}}\cap [\mathcal{L}, \mathcal{L}]$. If ${\rm char}(K)\ne 2$, since the element on the right side of equation (\ref{equation_first_1}) belongs to $I$ of $\mathcal{C}$, we get that $\displaystyle \sum_{i=1}^k 2a_i p_i p_i^* \in I$ or $\displaystyle \sum_{i=1}^k a_i p_i p_i^*$ is also in $I$. It follows that
            $x=\sum_{j=1}^{l'}b_i'(t_{i_j}q_{i_j}^*-q_{i_j}t_{i_j}^*)+\displaystyle \sum_{i=1}^m \displaystyle\sum_{j=1}^{m_i} c_{ij} \Big( x_{ij} y_{ij}^* x_{ij}^*- x_{ij} y_{ij} x_{ij}^*\Big)+I.$ The converse is clear.  
\end{proof}
			The following corollary follows immediately from Theorem \ref{theorem_3.6} and Lemma \ref{propostion_2.4}.
\begin{corollary}\label{corollary_3.7}
			Let $E$ be a graph and $K$ a field. If $E$ does not contain a cycle, then the following assertions hold:
			\begin{enumerate}[font=\normalfont]
				\item If ${\rm char}(K)\ne 2$ then $\mathbf{K}_{\mathcal{L}}\cap [\mathcal{L},\mathcal{L}]=\mathbf{K}_{\mathcal{L}}$.
				\item If ${\rm char}(K)=2$ then $x\in \mathbf{K}_{\mathcal{L}}\cap [\mathcal{L},\mathcal{L}]$ if and only there exist $k,l\in\mathbb{Z}_+$, $a_i,b_i\in K$, and $p_i, q_i, t_i\in{\rm Path}(E)$ such that 
				$$
				x=\displaystyle \sum_{i=1}^k a_i p_i p_i^* +\sum_{i=1}^{l}b_i(t_iq_i^*-q_it_i^*),
				$$
				where $\sum_{i=1}^ka_i\epsilon_{r(p_i)}\in{\rm span}_K\{B_v|v\in E^0\}$ and $t_i\ne q_i$ for all $1\leq i \leq l$. 
			\end{enumerate}
\end{corollary}

\begin{corollary}\label{corollary_3.8}
			Let $E$ be a graph and $K$ a field. If $E$ contains a cycle, then $\mathbf{K}_{\mathcal{L}}\ne[\mathbf{K}_{\mathcal{L}},\mathbf{K}_{\mathcal{L}}]$.
\end{corollary}
\begin{proof}
			View $\mathcal{L}$ as $\mathcal{C}/I$. Let $c$ be a cycle in $E$, and let $x=c-c^*\in \mathcal{C}$. It is clear that $x\notin I$. Let $\bar{x}$ denote $x+I$ in $\mathcal{L}$. Then, we get that $\bar{x}\ne\bar{0}$ in $\mathcal{L}$. Moreover, according to Proposition \ref{propostion_2.4} and Theorem \ref{propostion_2.4}, we conclude that $\bar{x}\in \mathbf{K}_{\mathcal{L}}$ but $\bar{x}\notin [\mathcal{L}, \mathcal{L}]$. This implies that $[\mathbf{K}_{\mathcal{L}},\mathbf{K}_{\mathcal{L}}]\ne \mathbf{K}_{\mathcal{L}}$.
\end{proof}
			The following two propositions illustrate that we can construct Leavitt path algebras $\mathcal{L}$ for which $\mathbf{K}_{\mathcal{L}} \neq [\mathbf{K}_{\mathcal{L}}, \mathbf{K}_{\mathcal{L}}]$, thereby providing counterexamples that resolve Question 1 in the negative.
\begin{proposition}\label{label_theorem_main}
            Let $\mathcal{L}$ be a simple Leavitt path algebra over a graph contaning a cycle. Then $\mathbf{K}_{\mathcal{L}}\ne [\mathbf{K}_{\mathcal{L}},\mathbf{K}_{\mathcal{L}}]$.
\end{proposition}

\begin{example}
        	Let $R_n$ be the rose with $n$ petals graph:
        	
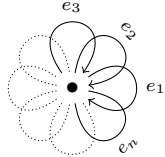
\begin{center}
			\begin{tikzpicture}
			%\node at (-2,0) (1) {$R_n:$};
			\node at (0,0) (0) {$\bullet$};
		
			\draw [->] (0) to [out=125,in=55,looseness=10] node[above] {$_{e_3}$} (0);
			\draw [->] (0) to [out=80,in=10,looseness=12] node[sloped,above] {$_{e_2}$} (0);
			\draw [->] (0) to [out=35,in=-35,looseness=9] node[right] {$_{e_1}$} (0);
			\draw [-] (0) to [out=-10,in=-80,looseness=12] node[sloped,below] {$_{e_n}$} (0);
			\draw [densely dotted] (0) to [out=-55,in=-125,looseness=10] node[above] {} (0);
			\draw [densely dotted] (0) to [out=-170,in=-100,looseness=12] node[above] {} (0);
			\draw [densely dotted] (0) to [out=215,in=145,looseness=9] node[above] {} (0);
			\draw [densely dotted] (0) to [out=170,in=100,looseness=12] node[above] {} (0);
			\end{tikzpicture}
  	      \end{center}
     	   Then $L_K(R_n)$ is a simple ring and $[\mathbf{K}_{L_K(R_n)},\mathbf{K}_{L_K(R_n)}]\ne \mathbf{K}_{L_K(R_n)}$.
\end{example}
			In \cite{Pa_hennig14}, J. Hennig provided an example of a simple, locally finite-dimensional associative algebra $\mathcal{A}$ over an algebraically closed field, equipped with an involution, such that $\mathbf{K}_{\mathcal{A}} \neq [\mathbf{K}_{\mathcal{A}}, \mathbf{K}_{\mathcal{A}}]$. In the following proposition, we utilize Leavitt path algebras to construct similar algebras with the same property:
\begin{proposition}
			There exist a simple, locally finite Leavitt path algebra $\mathcal{L}$ such that $\mathbf{K}_{\mathcal{L}}\ne [\mathbf{K}_{\mathcal{L}},\mathbf{K}_{\mathcal{L}}]$.
\end{proposition}
\begin{proof}
			Let $K$ be a field with ${\rm char}(K)=2$. Let $A_{\mathbb{Z}}$ be the following graph where $(A_{\mathbb{Z}})^0$ indexed by $\mathbb{Z}$:
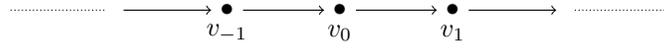
			\begin{center}
				\begin{tikzpicture}
					\node at (0,0) (1) {$  $};
					\node at (-1.5,0) (-2) {$ $};
					\node at (0,0) (1) {$ $};
					\node at (1.5,0) (2) {$\bullet$};
					\node at (1.5,-0.3) {$v_{-1}$};
					\node at (3,0) (3) {$\bullet$};
					\node at (3,-0.3) {$v_0$};
					\node at (4.5,0) (4) {$\bullet$};
					\node at (4.5,-0.3) {$v_1$};
					\node at (6,0) (5) {$ $};
					\node at (7.5,0) (6) {$ $};
					\draw [-, densely dotted] (-2) to node[above] {$ $} (1);
					\draw [->] (1) to node[above] {$ $} (2);
					\draw[->] (2) to node[above] {$ $} (3);
					\draw [->] (3) to node[above] {$ $} (4);
					\draw [->] (4) to node[above] {$ $} (5);
					\draw [-, densely dotted] (5) to node[above] {$ $} (6);
				\end{tikzpicture}
			\end{center}
			Let $\mathcal{L}=L_K(A_{\mathbb{Z}})$ Then, for any $v_i\in (A_{\mathbb{Z}})^0$, we have 
			$$
			B_{v_i}=(\dots, 0, \underset{\text{$i$-th}}{1},1,0\dots)\in \mathbb{Z}^{(\mathbb{Z})}.
			$$
			Let $\mathbb{V}={\rm span}_K\{B_v\mid v\in (A_{\mathbb{Z}})^0\}\subseteq\mathbb{Z}^{(\mathbb{Z})}$. Then, for any $0\ne B\in \mathbb{V}$, there exist integers $i_1<i_2<\cdots<i_n$ such that 
			\begin{align*}
				&	B			= a_{i_1}B_{v_{i_1}}+a_{i_2}B_{v_{i_2}}+\cdots+a_{i_n}B_{v_{i_n}} \\
				&\hspace*{0.27cm} = (\dots, 0, a_{i_1},\dots, a_{i_n+1},0, \dots ).
			\end{align*}
			In other words, any $0\ne B\in \mathbb{V}$ has at least two non-zero coordinates. Thus, for any $v\in E^0$, by (2)  of Corollary \ref{corollary_3.7}, we get $v\in\mathbf{K}_{\mathcal{L}}\backslash [\mathbf{K}_{\mathcal{L}},\mathbf{K}_{\mathcal{L}}]$. Because the only hereditary and saturated subsets of $A_{\mathbb{Z}}$ are $\varnothing$ and $(A_{\mathbb{Z}})^0$, we conclude that $ \mathcal{L}$ is simple. Moreover, according to \cite[Proposition 4.4]{Pa_khanh24}, we also get that $ \mathcal{L}$ is a locally finite $K$-algebra. \\
\section{Declarations}
{\noindent\textbf{Data Availability.}} The authors declare that no data has been used for this
research.
			
\bigskip 
			
{\noindent\textbf{Conflict of Interest.} }The authors declare that they have no conflict of interest.
			
\end{proof}

\end{document}